\documentclass{article}

\RequirePackage[OT1]{fontenc}
\RequirePackage{amsmath,amsfonts,amsthm}
\usepackage{authblk}
\newtheorem{theorem}{Theorem}

\newtheorem{lemma}{Lemma}
\newtheorem{proposition}{Proposition}
\newtheorem{remark}{Remark}
\DeclareMathOperator*{\argmax}{argmax}

\newcommand{\mJ}{{\mathcal J}}

\newcommand{\mS}{{\mathcal S}}

\newcommand{\bN}{{\mathbb N}}
\newcommand{\bR}{{\mathbb R}}
\newcommand{\bZ}{{\mathbb Z}}
\newcommand{\bP}{{\mathbb P}}
\newcommand{\bE}{{\mathbb E}}
\newcommand{\bI}{{\mathbb I}}

\newcommand{\coS}{<\!{\mathcal S}\! >}
\newcommand{\coqS}[1]{<\!{\mathcal S}\wedge Q(#1)\! >}

\begin{document}
%
%
%
%
 \title{Stability of MaxWeight-($\alpha$,$g$)}
\author[1]{Neil Walton}
\affil[1]{University of Amsterdam, n.s.walton@uva.nl}
\date{}
%
%


\maketitle

\abstract{
We consider a single-hop switched queueing network. 
Amongst a plethora of applications, these networks have been used to model wireless networks and input queued switches.
The MaxWeight scheduling policies have proved popular, chiefly, because they are throughput optimal and do not require explicit estimation of arrival rates.

In this article, we prove the same throughput optimality property for a generalization of the MaxWeight policy called the \emph{MaxWeight-($\alpha,g$) policy}. 
In brief, given parameter $\alpha>0$ and concave functions $g=(g_j$: $j\in\mJ)$, the MaxWeight-($\alpha,g$) policy chooses a random schedule which solves
\begin{align*}
&\text{maximize} &&\sum_{j\in\mJ} g_j(s_j)Q_j^\alpha\\
&\text{over} && s\in\coS,
\end{align*}
where $Q=(Q_j:j\in\mJ)$ is the vector of queue sizes and where, allowing for randomization, the maximization is taken over a set of admissible schedules $\coS$. These throughput optimal, myopic scheduling policies allow for scheduling choices similar to those found in bandwidth sharing networks -- a further well studied model of Internet congestion.
}

\section{Introduction}\label{sec1}
The MaxWeight scheduling policies were first introduced by Tassiulas and Ephremides \cite{TaEp92} as a model of wireless communication. Their policy was applicable to the class of switched queueing networks, where there are constraints on which queues can be served simultaneously.
Subsequently, as a model of Internet protocol routers, McKeown et al. \cite{MMAW99} applied this paradigm to the example of input-queued switches and coined the term \emph{throughput optimal}.

Roughly stated, a queue scheduling policy is throughput optimal if it is stable for every arrival rate for which there exists a stabilizing policy. If the vector of arrival rates is known then throughput optimality is trivial: at each time, one can chose a random schedule whose average service rate dominates each arrival rate. However, in practice, explicit knowledge of arrival rates is not achievable, particularly when rates may vary over time. A striking feature of the MaxWeight is that it is throughput optimal {and} it is \emph{myopic} -- meaning that only current queue state information is required to choose a schedule.  
For this reason, the MaxWeight policy has proved popular. Accordingly, the policy has been generalized.  Before proceeding with a formal analysis, let's informally discuss these generalizations; let's discuss Bandwidth Sharing Networks, a different but somewhat related model of Internet control; and let's discuss the MaxWeight-($\alpha ,g$) policy from which this paper is titled.

\subsubsection*{MaxWeight-$f$}
Currently, MaxWeight and its generalizations can be expressed in terms of the following MaxWeight-$f$ policy.
The MaxWeight-$f$ policy chooses a schedule at each time which optimizes
\begin{align*}
&\text{maximize} &&\sum_{j\in\mJ} s_jf_j(Q_j)\\
&\text{over} && s\in\coS.
\end{align*}
Here $Q_j$ is the queue sizes for each queue $j\in\mJ$, and $s_j$ is the mean number of queue $j$ jobs served under a random schedule in $\coS$. The choice of $\coS$ allows for randomization; nonetheless, randomization needn't be used as there always exists a deterministic solution to the above optimization. 

The MaxWeight policy introduced by Tassiulas and Ephremides \cite{TaEp92} is the case where $f$ is the identity function. A broad class of functions $f$ are proven to be throughput optimal by Meyn \cite{Me09}. Amongst these choices, the MaxWeight-$\alpha$ policies have proved popular \cite{AKRSVW04,St04,ShWi06,ShWi11,JMMT11,ShWi12}. Here, each $f_j$ is a power function: $f_j(s_j)=s_j^\alpha$.

So, thus far, generalizations of the MaxWeight policy alter the dependence on queue sizes, $Q_j$, with an appropriately chosen function, say $f_j(Q_j)=Q_j^\alpha$.


\subsubsection*{Bandwidth Sharing Networks}
Somewhat separate to this work on switch networks, researchers have considered bandwidth sharing networks \cite{MaRo99,VLK01,BoMa01,KeWi04,EBZ07,GoWi09,AyMa09,Le12,PTFA12}. These continuous-time Markov processes form a flow level model of Internet file transfer. 
Here, using slightly unconventional notation, it is assumed that the bandwidth allocation achieved by a congestion control protocol solves a \emph{network utility optimization} of the form
\begin{subequations}\label{netutil}
\begin{align}
&\text{maximize} &&\sum_{j\in\mJ} g_j\Big(\frac{s_j}{Q_j}\Big) Q_j\\
&\text{over} && s\in\coS.
\end{align}
\end{subequations}
Here $Q_j$ represents the queue size of class $j$ jobs to be transfered across the network and $s_j$ is the bandwidth allocated to class $j$. It is assumed that the allocation belongs to some convex region $\coS$ and that $g_j:\bR_+\rightarrow \bR$ is an increasing concave function.

Switch networks serve jobs one-by-one subject to scheduling constraints in a first-in-first-out manner. By contrast, bandwidth sharing networks serve jobs simultaneously subject to rate constraints in a processor-sharing manner. Bandwidth sharing networks have been shown to be throughput optimal in a number of different settings, \cite{MaRo99,VLK01,BoMa01,Ma07}. 

So, generalizations of the MaxWeight policy alter the dependence on queue sizes, $Q_j$, with an appropriately chosen function, $f_j$. By contrast, bandwidth sharing networks alter the dependence on schedules, $s_j$, with an appropriately chosen function, $g_j$.
It would seem possible that one could form scheduling policies for switch networks based on network utility optimization, as described above. To the knowledge of this author, discussions to this end are first made by Wischik and Shah \cite{ShWi11} and Zhong \cite{Zh12} for the weighted $\alpha$-fair bandwidth allocation policies \cite{MoWa00}.

\subsubsection*{MaxWeight-$(\alpha,g)$}
We will discuss precise terms shortly, but, in brief, the MaxWeight-($\alpha,g$) policy chooses a schedule which optimizes
\begin{align*}
&\text{maximize} &&\sum_{j\in\mJ} g_j(s_j)Q_j^\alpha\\
&\text{over} && s\in\coS,
\end{align*}
where $Q_j\in\bZ$ is the queue sizes for each queue $j\in\mJ$ and where $s_j$ is the mean number of jobs served under random schedule in $\coS$. Here, for $j\in\mJ$, $g_j$ is a convex increasing function and $\alpha$ is a positive real number.

The MaxWeight-$(\alpha,g)$ policy generalizes the MaxWeight-$\alpha$ policy, it is myopic, and it incorporates a functional dependence similar to the network utility optimization policies, described above. As we will discuss, they allow the scheduling analogue of weight $\alpha$-fair bandwidth allocation \cite{MoWa00}.
 Although it is clear we can define a MaxWeight-$(f,g)$ schedule where we replace the function $Q_j^\alpha$ with a function $f_j(Q_j)$, we do not provide an analysis of this policy.

Our stability analysis for these policies can be further justified because it is conjectured that such policies can have optimal behaviour. In particular, Zhong \cite[pg. 102]{Zh12} conjectures the proportional fair scheduling policy -- where $\alpha=1$ and $g_j(s)=\log(s)$ -- has optimal queue size scaling in heavy traffic for an input-queued switch. Here scaling is taken both with respect to traffic intensity and network size. 

In this paper, we consider the MaxWeight-($\alpha,g$) as a generalized policy for switch networks. So why not consider it as an appropriate generalization for bandwidth networks? This is because the optimal network utility maximization \eqref{netutil} is not a scheduling decision but an equilibrium reached by competing Internet flows. Thus one must first argue that the MaxWeight-($\alpha,g$) objective is achieved by a congestion control protocol  before such an analysis.\vspace{0.3cm}

\subsubsection*{Contribution}
The main contribution of this paper is to prove that the MaxWeight-($\alpha,g$) policy is throughput optimal under independent identically distributed arrivals.

For MaxWeight, the throughput optimality is the foremost property associated with this myopic policy.
So, although there are a raft of other open problems which one could ask, in this article we focus on addressing this first order stability question.

In the conclusions, we will discuss a number of other open questions associated with the MaxWeight-($\alpha,g$)  policy some of which have been answered for the MaxWeight policy and for Bandwidth Sharing Networks.

\section{Related work}\label{sec2}
As discussed, the MaxWeight policy was first defined by Tassiulas and Ephremides \cite{TaEp92} and then considered in the context of input queued switches by McKeown et al. \cite{MMAW99}. The generalized MaxWeight policies are subsequently analyzed by a number of authors \cite{AKRSVW04,St04,DaLi05,Me09}. The behaviour of MaxWeight has been studied in a number of limiting regimes: fluid limits \cite{DaPr00}, heavy traffic \cite{St04}, large deviations \cite{Su10,Su11}, and overload \cite{ShWi11}. Notably, Shah and Wischik \cite{ShWi11} in their study of overloaded networks analyse both bandwidth sharing networks and switch networks together and, to the best of knowledge of this author, they are this first to suggest using a network utility function for switch scheduling. Further work combining switched networks with bandwidth sharing can be found in Moallemi and Shah \cite{MoSh10},  Zhong \cite{Zh12}, and Shah et al. \cite{ShWaZh12}.

It has been argued that congestion control protocols implicitly solve a network utility optimization \cite{Ke97,Sr04}. This led to bandwidth sharing networks \cite{MaRo99}, which model the stochastic arrival and departure of utility optimizing Internet flows. Analogous to MaxWeight, there are numerous stability proofs \cite{VLK01,BoMa01,Ye03,GoWi09,Le12,PTFA12}, and these policies have similarly been analyzed in different limiting regimes: fluid limits \cite{KeWi04}, heavy traffic \cite{KKLW09}, large deviations \cite{Ma07}, and overload \cite{EBZ07}.

We prove the stability of the MaxWeight-$(\alpha,g)$ policy using method of fluid limits adopted by Rybko and Stolyar \cite{RySt92}, Dai \cite{Da95,Da95b}, and Bramson \cite{Br08}. This method is applied to the MaxWeight policy by \cite{DaPr00}. However, our proof more closely follows the fluid analysis of Kelly and Williams \cite{KeWi04} and stability proof of Bonald and Massoulie \cite{BoMa01} which are applied to bandwidth sharing networks.

MaxWeight policies apply to switched networks and authors have investigated numerous applications. These include wireless networks \cite{TaEp92}, input queued switches \cite{MMAW99}, data centers \cite{ShWi11}, semiconductor wafer fabrication \cite{DaLi05}, road traffic management \cite{Va09}, on-line auctions \cite{TaSr12} and call centers \cite{MaSt04}.

\section{Structure}\label{sec3}
In the next section, we introduce the main notation used throughout the paper. In Section \ref{sec5}, we formally define the MaxWeight-$(\alpha,g)$ policy and we state the main result of the paper: Theorem \ref{MainThrm}. We prove Theorem \ref{MainThrm}, by analyzing the stability of an appropriate fluid model. We define this fluid model in Section \ref{sec6}. We, also, state Proposition \ref{FluidLimThrm}. This fluid limit result justifies the fluid model's connection with the MaxWeight-$(\alpha,g)$ queueing process. The proof of Proposition \ref{FluidLimThrm} is given in the Appendix. In Section \ref{sec7}, we prove the stability of our fluid model. This proof is the crux of the paper. Following this, in Section \ref{sec8}, we combine Propositions \ref{FluidLimThrm} and \ref{FluidStabProof}  to prove the main result, Theorem \ref{MainThrm}. We then conclude the paper discussing possible lines of further research.

\section{Queueing Network Notation}\label{sec4} 
We define a discrete-time queueing network where there are restrictions on which
queues can be served simultaneously. 

We let the finite set $\mJ$ index the set of queues. 
We let $a(t)=(a_j(t):j\in\mJ)\in\bZ_+^\mJ$ be the number of arrivals occurring at each queue
at time $t\in\bN$. We assume $\{ a(t) \}_{t=1}^\infty$ is a sequence of independent identically distributed random vectors 
with finite mean $\bar{a}\in(0,\infty)^\mJ$ and finite variance $\bE\big[ a_j(t)^2\big] \leq K$ for all $j\in\mJ$.\footnote{The assumption of finite variance is not essential; however, it allows bounds more convenient for our proofs.}

We let the finite set $\mS$ be the set of schedules. 
Each $\sigma=(\sigma_j:j\in\mJ)\in\mS$ is a vector in $\bZ_+^\mJ$ where $\sigma_j$ gives the number of jobs that will be served from queue $j$ under schedule $\sigma$. We assume the zero vector, $\textbf{0}$, belongs to $\mS$.
We let $\coS$ be the convex combination of schedules in $\mS$. We assume $\coS$ has a non-empty interior.
Note the extreme points of $\coS$ are a subset of $\mS$. 
So any point in $\coS$ can be expressed as a convex combination of points in $\mS$. 

For $q,s\in\bR$, we define $q\wedge s=\min\{q,s\}$. If $q,s\in\bR^\mJ$ then $q\wedge s= (q_j\wedge s_j: j\in\mJ)$.
For a vector $Q\in\bZ_+^\mJ$, we define
\begin{align*}
\mS\wedge Q = \{ \sigma \wedge Q : \sigma \in \mS \}.
\end{align*}
We let $<\!\mS\wedge Q\! >$ be the convex combination of points in $\mS \wedge Q$. 
Notice, to prevent a queue having a negative number of jobs, $\sigma\wedge Q$ is the number of jobs that would be served under schedule $\sigma$ when $Q$ gives the vector of queue sizes.

We let $Q(0)=(Q_j(0):j\in\mJ)$ be the number of jobs in each queue at time $t=0$.
From a sequence of schedules $\{ \sigma(t) \}_{t=1}^\infty$, we can define the queue size vector $Q(t)=(Q_j(t):j\in\mJ)$ by
\begin{equation}
Q_j(t+1)= Q_j(t) - \sigma_j(t+1)  + a_j(t+1),
\end{equation}
for $j\in\mJ$, and $t\in \bN$. Recall, for positive queue sizes, it is required that $$\sigma_j(t+1)\in \mS \wedge Q(t),$$
for all $t\in\bZ_+$.\footnote{Observe, this choice of notation is equivalent to defining the queueing process by $Q_j(t+1)= [ Q_j(t) - \sigma_j(t+1)]_+  + a_j(t+1)$.} 

Given $(Q(t): t\in\bN)$ defines a Markov chain, we say the queue size process is throughput optimal if it is positive recurrent whenever the vector of arrival rates, $\bar{a}$, belongs to the interior of $\coS$.

\section{MaxWeight-($\alpha$,$g$)}\label{sec5}
We now define the MaxWeight-$(\alpha,g)$ policy.
We let $\alpha$ be a positive real number and, for each $j\in\mJ$, we let $g_j:\bR_+\rightarrow \bR$  be a strictly increasing, differentiable, strictly concave function. Given a queue size vector $Q(t-1)=(Q_j(t-1):j\in\mJ)$ at time $t-1$, we define $\bar{\sigma}(t)=(\bar{\sigma}_j(t):j\in\mJ)$ to be a solution to the optimization
\begin{subequations}\label{MW-af}
\begin{align}
&\text{maximize} &&\sum_{j\in\mJ} g_j(s_j)Q_j(t-1)^\alpha\label{MW-af-1}\\
&\text{over} && s\in\coqS{t-1}.\label{MW-af-3}
\end{align}
\end{subequations}
 
In general, $\bar{\sigma}(t)$ need not belong to the set of schedules $\mS\wedge Q(t-1)$. However, $\bar{\sigma}(t)$ is a convex combination of points in $\coqS{t-1}$. Thus we let $\sigma(t)$ be a random variable with support on $\mS\wedge Q(t-1)$ and mean $\bar{\sigma}(t)$. The MaxWeight-$(\alpha,g)$ scheduling policy is the policy that chooses schedule $\sigma(t)$ at each time $t\in\bN$.

To be concrete, we assume the random variables $\sigma(t)$ are, respectively, a function of $Q(t-1)$ and an independent (uniform) random variable. Expressed differently, $\bE \left[ \sigma(t) | Q(t-1) \right]$, $t=1,2,...$, are independent random variables. This ensures that 
the queue size process $\{ Q(t)\}_{t=0}^\infty$ associated with the MaxWeight-$(\alpha,g)$ scheduling policy is a discrete-time Markov chain. Further, the constraints \eqref{MW-af-3} ensure that a schedule never exceeds the queue it serves. This said, the constraint induced by queue size vector $Q(t-1)$ is only relevant when a queue may empty and, indeed, could be chosen in other ways for these instances. 

For later use, for $q\in\bR_+^\mJ$, we let $\hat{\sigma}(q)$ be the solution to the optimization problem
\begin{subequations}\label{MW-qq}
\begin{align}
&\text{maximize}&& \sum_{j\in\mJ} g_j(s_j) q_j^{\alpha}\label{MW-qq-1}\\
&\text{over}&& s\in <\! \mS \wedge q\! >.\label{MW-qq-2}
\end{align}
\end{subequations}

\begin{remark}
We remark that the MaxWeight-$(\alpha,g)$ policy coincides with the MaxWeight-$\alpha$ policy when $g_j(s_j)=s_j$ and with the MaxWeight policy for $\alpha=1$ when $g_j(s_j)=s_j$. An addition class of further interest is the $\alpha$-fair policies which, for $\alpha>0,$ is defined from the maximization
\begin{align*}\label{MW-ab}
&\max_{s\in\coqS{t-1}} \sum_{j\in\mJ}\frac{ s_j^{1-\alpha}}{1-\alpha}Q_j(t-1)^{\alpha}, && \text{if } \alpha\neq 1,\\
&\max_{s\in\coqS{t-1}}\sum_{j\in\mJ} \log(s_j) Q_j(t-1), && \text{if } \alpha=1.
\end{align*}
As discussed in the introduction, the $\alpha$-fair policies were first introduced and studied in the context of Bandwidth Sharing Networks, \cite{MoWa00,BoMa01}.
For switch-networks, the $\alpha$-fair policy is first defined by Shah and Wischik \cite{ShWi11}.
The $\alpha=1$ policy is referred to as the proportionally fair policy \cite{Ke97}. For input-queued switches, it is conjectured that the proportionally fair switch scheduling policy has optimal queue size scaling in heavy traffic \cite{Zh12}.
\end{remark}

The main result of this article is the following proof of throughput optimality.

\begin{theorem}\label{MainThrm}
For the MaxWeight-$(\alpha,g)$ policy,
if the vector of average arrival rates $\bar{a}$ belongs to the interior of $\coS$
then the queue size process $(Q(t), t\in\bZ_+)$  is positive recurrent. 
\end{theorem}

\section{Fluid model}\label{sec6}
In this section, we state a fluid model and a fluid limit associated with the MaxWeight-$(\alpha,g)$ policy. 
The proof of the fluid limit is given in Appendix \ref{FluidLimit}.

A positive, absolutely continuous process $q(t)= (q_j(t): j\in\mJ)$, $t\in\bR_+$, is a fluid model for the MaxWeight-$(\alpha,g)$ policy if, for $j\in\mJ$ and for almost every\footnote{By \emph{almost every}, we mean on all points except for a set of Lebesgue measure zero.} time $t\in\bR_+$,
\begin{subequations}\label{FluidEqns}
\begin{align}
\frac{d q_j}{dt} & = \bar{a}_j - {\sigma}^*_j(q(t)), && \text{if } q_j(t)>0,\\
\frac{d q_j}{dt} & = 0, && \text{if } q_j(t)=0,
\end{align}
\end{subequations}
where ${\sigma}^*(q)$ solves the optimization
\begin{subequations}\label{MW-sq}
\begin{align}
&\text{maximize}&&\sum_{j\in\mJ} g_j(s_j)q_j^\alpha\label{MW-sq-1}\\
&\text{over}&& s\in\coS.\label{MW-sq-2}
\end{align}
\end{subequations}
The key distinction between ${\sigma}^*(q)$, above, and $\hat{\sigma}(q)$, defined by optimization \eqref{MW-qq}, is that we remove the constraint where we cannot schedule more jobs than there are in each queue, i.e. compare \eqref{MW-qq-2} with \eqref{MW-sq-2}. 

We can formalize the sense that $q(t)$ is the limit of the MaxWeight-$(\alpha,g)$ policy. We let $\{ Q^{(c)} \}_{c\in\bN}$ be a sequence of versions of our queueing process for the MaxWeight-$(\alpha,g)$, where $||Q(0)||_1=c$. Here and hereafter, $||\cdot||_1$ is the $L_1$-norm. We define
\begin{align}
\bar{Q}^{(c)}(t)& =\frac{Q^{(c)}(\lfloor ct\rfloor)}{c},
\end{align}
for $t\in\bR_+$.  The following result formalizes a fluid model $q$, \eqref{FluidEqns}, as the limit of $\{ \bar{Q}^{(c)} \}_{c\in\bN}$. In informal terms, it states that the only possible limit of the sequence $\{ \bar{Q}^{(c)} \}_{c\in\bN}$ as $c\rightarrow\infty$ is a process $q$ satisfying \eqref{FluidEqns}.

\begin{proposition}[Fluid Limit]\label{FluidLimThrm}
The sequence of stochastic processes $\{ \bar{Q}^{(c)}\}_{c\in\bN}$ is tight\footnote{Recall a sequence of random processes $\{ \bar{Q}^{(c)}\}_{c}$ is tight if every subsequence of $\{ \bar{Q}^{(c)}\}_{c}$ has a weakly convergent subsequence.} with respect to the topology of uniform convergence on compact time intervals. Moreover,  any weakly convergent subsequence of  $\{ \bar{Q}^{(c)}\}_{c\in\bN}$ converges to a Lipschitz continuous process satisfying fluid equations \eqref{FluidEqns}.
\end{proposition}

\noindent The proof --and indeed statement-- of Proposition \ref{FluidLimThrm} is somewhat technical. However, the main point is that we can compare the max-weight queueing process to a tractable fluid model, $q$. Proposition \ref{FluidLimThrm}  is proven in Appendix \ref{FluidLimit}. We now analyse the stability of the fluid model $q$.

\section{Proof of Fluid Stability}\label{sec7}
We now consider a process $q(t)$, $t\in\bR_+$, that satisfies the fluid limit equations \eqref{FluidEqns}. In the following theorem, we show that these fluid solutions are stable in the sense that they hit the zero state in finite time. This result will be sufficient to prove positive recurrence of the MaxWeight-$(\alpha,g)$ queue size process.

 \begin{proposition}[Fluid Stability]\label{FluidStabProof}
There exists a time $T>0$ such that,
for every fluid model $(q(t): t\in\bR_+)$ satisfying \eqref{FluidEqns} and with $||q(0)||_1=1$,
\begin{align}
&\qquad\qquad\qquad q_j(t)=0, && j\in\mJ, 
\end{align}
{ for all } $t\geq T$.
\end{proposition}

The main idea is to consider the gradient of the tangent line of the MaxWeight-$(\alpha,g)$ objective, \eqref{MW-sq-1}, between two points: the arrival rate and the optimal solution, see \eqref{rhocond} below. Integrating this obtains a Lyapunov function \eqref{MW-Lya}. This idea is used by Bonald and Massouli\'e \cite{BoMa01} in their analysis of weighted $\alpha$-fair bandwidth sharing networks. For switch networks, Proposition \ref{FluidStabProof} follows analogously.

\begin{proof}
We define $G_q(s)$ to be the objective of the MaxWeight-$(\alpha,g)$ optimization,
\begin{equation}
G_q(s)=\sum_{j\in\mJ} g_j(s_j) q_j^\alpha.
\end{equation}•
Recall that in our fluid equations
\begin{equation*}
{\sigma}^*(q(t))\in\argmax_{s\in\coS}\; G_{q(t)}(s).
\end{equation*}
Any vector $\rho$ belonging to the interior of $\coS$ is not optimal. Thus  $G_{q(t)}(\rho)< G_{q(t)}({\sigma}^*(q(t)))$. As $G_{q(t)}(\cdot)$ is strictly concave,  $G_{q(t)}(\cdot)$ must be increasing along the line connecting $\rho$ to ${\sigma}^*(q(t))$. In other words, for all $\rho\in\coS^\circ$ and for  $q(t)\neq 0$,
\begin{equation}\label{rhocond}
\Big({\sigma}^*(q(t))-\rho\Big) \cdot \nabla G_{q(t)}(\rho) > 0.
\end{equation}
Here $\nabla G_{q}(\rho)= ( g'_j(\rho_j) q_j(t)^\alpha : j\in\mJ)$.  
Since $\bar{a}$ belongs to the interior of $\coS$, there exists $\epsilon>0$ such that $ (1+\epsilon) \bar{a}\in \coS$. We define $\rho=(1+\epsilon)\bar{a}$. In this case, we can re-express the inequality \eqref{rhocond} as follows
\begin{equation}\label{ineq}
\sum_{j\in\mJ} \Big(\bar{a}_j-{\sigma}_j^*(q(t)) \Big) g_j'(\rho_j) q_j(t)^\alpha \leq -\epsilon \sum_{j\in\mJ} g_j'(\rho_j) q_j(t)^\alpha .
\end{equation}
As $\frac{d q_j}{dt}=\bar{a}_j-{\sigma}_j^*(q(t))$, we define the Lyapunov function
\begin{equation}\label{MW-Lya}
L(q)= \sum_{j\in\mJ} g'(\rho_j) \frac{q_j^{1+\alpha}}{1+\alpha},
\end{equation}
$q\in\bR_+^\mJ$. The function $L(q)$ is positive and $L(q)=0$ iff $q_j=0$ for all $j\in\mJ$. We now observe
\begin{align}
\frac{dL(q(t))}{dt}&=\sum_{j\in\mJ} \Big(\bar{a}_j-{\sigma}_j^*(q(t)) \Big)g_j'(\rho_j) q_j(t)^\alpha \notag\\
& \leq  -\epsilon \sum_{j\in\mJ} g_j'(\rho_j) q_j(t)^\alpha. \label{diffinequ}
\end{align}
The equality holds by the chain rule and the inequality holds by \eqref{ineq}.

We define norms on $\bR_+^\mJ$
\begin{align}
& ||q||_{1+\alpha} = ( L(q) )^{\frac{1}{1+\alpha}},\\
& || q||_{\alpha} = \Big( \sum_{j\in\mJ} g_j'(\rho_j) q_j(t)^\alpha\Big)^{\frac{1}{\alpha}}.
\end{align}
By the Lipschitz equivalence of norms, there is a constant $\gamma>0$ such that $$ \gamma ||q||_{1+\alpha} \leq ||q||_\alpha,$$ for all $q\in\bR_+^\mJ$.\footnote{ Note,  $||q||_{1+\alpha} \leq (1+\alpha)^{-\frac{1}{1+\alpha}}|\mJ| \max_j g'_j(\rho_j) q_j $ and also note that  $\max_j g'_j(\rho_j) q_j\leq ||q||_{\alpha}$. So, for instance, we can take $\gamma=(1+\alpha)^{\frac{1}{1+\alpha}} | \mJ|^{-1}$.} Applying this observation to the inequality \eqref{diffinequ}, we see that
\begin{equation}\label{Lineq}
\frac{dL(q(t))}{dt} \leq -\epsilon \gamma^{\alpha}  L(q(t))^{\frac{\alpha}{1+\alpha}}.
\end{equation}
Observe, by the above inequality, if $L(q(T))=0$ for any differentiable point $T$ then  $L(q(t))=0$ for all $t\geq T$. Now, whilst $L(q(t))>0$, we have from \eqref{Lineq} that
\begin{align*}
&L(q(t))^\frac{1}{1+\alpha}-L(q(0))^\frac{1}{1+\alpha} \\
=& \int_0^t (1+\alpha)^{-1} L(q(t))^{\frac{-\alpha}{1+\alpha}}\frac{dL(q(t))}{dt} dt\\
\leq &  -\epsilon (1+\alpha)^{-1} \gamma^{\alpha} t
\end{align*}
Rearranging this expression, we see that for all times $t$
\begin{equation}
L(q(t))\leq  \left( L(q(0))^\frac{1}{1+\alpha}-\epsilon (1+\alpha)^{-1} \gamma^{\alpha} t \right)_+^{1+\alpha}.
\end{equation}
The function $L(q)$ is continuous and therefore bounded above by a constant, $K$, for all values of $q$ with $||q||_1=1$. Hence, if $||q(0)||_1=1$, $L(q(t))=0$ for all $t\geq T$ where 
\begin{equation*}
T=\frac{ (1+\alpha) K^{\frac{1}{1+\alpha}}}{\epsilon\gamma^{\alpha}},
\end{equation*}
and thus, as required, $q_j(t)=0$, $j\in\mJ$, for all $t\geq T$.
\end{proof}

\section{Proof of Positive Recurrence}\label{sec8}
We are now in a position to combine Propositions \ref{FluidLimThrm} and \ref{FluidStabProof} to prove Theorem \ref{MainThrm}. We could at this point apply the general stability results of Dai \cite{Da95,Da95b} and Bramson \cite{Br08}. However, for completeness we provide a self-contained proof.

\begin{proof}[Proof of Theorem \ref{MainThrm}]
For every $t\geq 0$, the sequence queue sizes $\{ \bar{Q}^{(c)}(t)\}_{c\in\bN}$ is uniformly integrable. This is proven in Lemma \ref{UIlemma} in Appendix \ref{Lemmas}. By Proposition \ref{FluidLimThrm}, for any unbounded sequence in $\bN$, there is a subsequence $\{c_k\}_{k\in\bN}$ for which $\bar{Q}^{(c_k)}$ converges in distribution to fluid solution ${q}$. Let $T$ be the time given in Proposition \ref{FluidStabProof}, where $q_j(T)=0$ for $j\in\mJ$. 
Since $\{|\bar{Q}^{(c_k)}(T)|\}_{c_k}$ is uniformly integrable and converges in distribution to ${q}(T)$, we also have $L_1$ convergence
\begin{equation}\label{fosters:fl1}
\lim_{c_k\rightarrow\infty} \bE ||\bar{Q}^{(c_k)}(T)||_1 = \bE ||q(T)||_1=0.
\end{equation}
This implies there exists a $\kappa$ such that for all $c> \kappa$
\begin{equation}\label{fosters:fl2}
\bE ||\bar{Q}^{(c)}(T)||_1  < (1-\epsilon).
\end{equation}
Note that if \eqref{fosters:fl2} did not hold then we could find a subsequence for which \eqref{fosters:fl1} did not hold; thus, we would have a contradiction. Expanding this inequality \eqref{fosters:fl2}, we have as described by Bramson \cite{Br08}, the following \emph{multiplicative Foster's condition}: for $\big|\big|{Q}(0)\big|\big|_1 > \kappa$
\begin{equation}\label{FosterCond}
\bE\left[ \big|\big|{Q}(T||{Q}(0)||_1)\big|\big|_1-\big|\big|{Q}(0)\big|\big|_1 \Big|  {Q}(0) \right] <  -\epsilon \big|\big|{Q}(0)\big|\big|_1 .
\end{equation}

We now use this to prove positive recurrence of the event $\left\{\big|\big|{Q}(t)\big|\big|_1 \leq \kappa\right\}$.
We consider our Markov chain, $Q$, at specific stopping times: $\tau_0:=0$ and for $n\in\bN$,
\begin{align*}
\tau_n &:=\tau_{n-1} + T||Q(\tau_{n-1})||_1,  &&\text{if } ||Q(\tau_{n-1})||_1>\kappa,\\
\tau_n &:=\tau_{n-1} + 1, &&\text{if }  ||Q(\tau_{n-1})||_1\leq \kappa,
\end{align*}
We define the stopping time $N=\min\{ n :\tau_n\geq t\}$.
Now,
\begin{align}
0 \leq & \bE ||Q(\tau_N )||_1  \notag\\
=& \bE ||Q(0)||_1 + \bE\left[ \sum_{n=1}^N  ||Q(\tau_n)||_1 -||Q(\tau_{n-1})||_1\right]\notag \\
%
%
\leq &  \bE ||Q(0)||_1 -\epsilon \bE\Big[ \sum_{n=1}^N  ||Q(\tau_{n-1})||_1 \bI\big[||Q(\tau_{n-1})||_1 > \kappa \big] \Big]\notag\\
&+ ||\bar{a}||_1 \bE\Big[ \sum_{n=1}^N \bI\big[||Q(\tau_{n-1})||_1 \leq \kappa \big] \Big] \label{Ineq1}\\
= &  \bE ||Q(0)||_1 -\frac{\epsilon}{T} \bE\Big[ \sum_{n=1}^N  \big(\tau_n-\tau_{n-1}\big) \bI\big[||Q(\tau_{n-1})||_1 > \kappa \big] \Big]\notag\\
&+ ||\bar{a}||_1 \bE\Big[ \sum_{n=1}^N \bI\big[||Q(\tau_{n-1})||_1 \leq \kappa \big] \Big]\label{Ineq2}\\
= &  \bE ||Q(0)||_1 -\frac{\epsilon}{T} \bE\big[ \tau_N \big]\notag\\
&+ \left( ||\bar{a}||_1+\frac{\epsilon}{T} \right)\bE\Big[ \sum_{n=1}^N \bI\big[||Q(\tau_{n-1})||_1 \leq \kappa \big] \Big]\label{Ineq3}\\
\leq &  \bE ||Q(0)||_1 -\frac{\epsilon}{T} t\notag\\
&+ \left( ||\bar{a}||_1+\frac{\epsilon}{T} \right)\bE\Big[ \sum_{s=1}^t \bI\big[||Q(s)||_1 \leq \kappa \big] \Big] \label{Ineq4}
\end{align}
In inequality \eqref{Ineq1},  we condition on the event  $\{ ||Q(\tau_{n-1})||_1 > \kappa \}$ and apply the multiplicative Foster's condition \eqref{FosterCond}; for equality \eqref{Ineq2}, we observe that, by definition, $\tau_n -\tau_{n-1} = T||Q(\tau_{n-1})||_1$ and then adding appropriate terms we get an interpolating sum for equation \eqref{Ineq3}; and finally, for \eqref{Ineq4}, we observe the sequence $||Q(\tau_n)||_1$, for $\tau_n<t$, must have hit below $\kappa$ less times than the sequence $||Q(s)||_1$, $s\leq t$.

Rearranging the positive expression \eqref{Ineq4}, dividing by $t$ and taking limits, we gain
\begin{equation*}
\liminf_{t\rightarrow \infty} \frac{1}{t} \bE\Big[ \sum_{s=1}^t \bI\big[||Q(s)||_1 \leq \kappa \big] \Big] \geq \frac{\epsilon }{T||a||_1+\epsilon}>0.
\end{equation*}
Thus, $Q$ is positive recurrent because it is positive recurrent in the finite set of states $\{q\in\bZ_+^\mJ: ||q||_1\leq \kappa\}$. 
\end{proof}

\section{Conclusion}\label{sec9}
We have now shown that the MaxWeight-($\alpha,g$) policy is throughput optimal. There are many other questions which could be asked for these policies. 

The first of which was the conjecture of Zhong \cite{Zh12} on the queue size scaling of the proportionally fair scheduling policy. A proof of this optimal behaviour further necessitates extensions beyond MaxWeight. More immediately, one could prove throughput optimality of MaxWeight-$(f,g)$, in a similar manner to that considered by Meyn \cite{Me09}. The analogous BackPressure policy could be considered in a similar manner to \cite{ShWi11}. Given that arrival rates may vary over time, one could further prove universal stability as considered by Neely \cite{Ne10}. A number of regimes beyond the fluid limit in this paper could be considered: heavy traffic \cite{St04}, overload \cite{ShWi11}, large deviations \cite{Su10,Su11}. Further, we do not analyse the computational complexity resolving scheduling solutions. Efficient solution of the MaxWeight-$(\alpha,g)$ optimization is of course important to any practical implementation \cite{Ta98,ShTsTs11}. We do not consider the effect of changes in network topology on stability \cite{VBY11}. Further, we do not consider decentralized implementation of this policy which has been a topic of recent investigation \cite{ShSh12,JiWa10}. 

Although there are many avenues that may be further pursued, we have first affirmatively answers this fundamental stability question for the MaxWeight-($\alpha,g$) polices.

\appendix
\section{Fluid Limit}\label{FluidLimit}

\begin{proof}[Proof of Proposition \ref{FluidLimThrm}]
To prove the tightness of a sequence of process $\{ X^{(c)} \}_c$, from Robert \cite[Theorem C.9]{Ro10}, 
 we see that we must prove
\begin{equation}\label{tightness}
\lim_{\delta\rightarrow 0} \bP\Bigg( \sup_{\substack{u,v: u,v<t\\|v-u|<\delta}}  \Big|\Big|X^{(c)}(v)-X^{(c)}(u) \Big|\Big|_1\geq \eta\Bigg)  =0.
\end{equation}
We wish to demonstrate this for $\{ \bar{Q}^{(c)} \}_c$.
For $j\in\mJ$,  we can express $\bar{Q}^{(c)}_j(t)$ as follows
\begin{align}
\bar{Q}^{(c)}_j(t) &=  \bar{Q}^{(c)}_j(0) + \bar{a}_j\frac{\lfloor ct \rfloor }{c} - \bar{S}^{(c)}_j(t) + \bar{M}^{(c)}_j(t) + \bar{N}^{(c)}_j(t) 
\end{align}
where, from $a(t),\bar{a},\sigma(t),\bar{\sigma}(t)$, we define
\begin{align*}
\bar{S}^{(c)}_j(t) &= \frac{1}{c}\sum_{s=1}^{\lfloor ct\rfloor} \bar{\sigma}_j(s),\\
&= \int_0^{{{\lfloor ct\rfloor} }/{c}} \hat{\sigma}_j({Q}^{(c)}_j(cs)) ds ,\\
  \bar{M}^{(c)}_j(t) &=\frac{1}{c}\sum_{s=1}^{\lfloor ct\rfloor}  \Big( a_j(s)-\bar{a}_j \Big),\\
\bar{N}^{(c)}_j(t) &= \frac{1}{c}\sum_{s=1}^{\lfloor ct\rfloor}  \Big(  \bar{\sigma}_j(s) - {\sigma}_j(s) \Big),
\end{align*}
for $j\in\mJ$. 

It is clear $Q_j(0) + \bar{a}_j{\lfloor ct\rfloor} /c$ satisfies \eqref{tightness}.  By the triangle inequality, it is sufficient to show \eqref{tightness} holds for each of the above terms. 

The term $\bar{\sigma}(s)$ is bounded by some constant $K_1$ for all $s$. So, we have the Lipschitz condition $$|\bar{S}_j(t)-\bar{S}_j(s)| < K_1|t-s| + 2K_1c^{-1}$$ which satisfies \eqref{tightness} with $\delta < \eta K_1^{-1}$. 

The process $\bar{M}_j(t)$ is a martingale thus by Doob's $L_2$ inequality
\begin{align}\label{tightM}
&\bP\Bigg( \sup_{\substack{u,v: u,v<t\\|v-u|<\delta}}  \Big| \bar{M}_j^{(c)}(v)-\bar{M}_j^{(c)}(u) \Big|\geq \eta\Bigg) \\
\leq &\bP \left( \sup_{u: u <t} \big| \bar{M}_j^{(c)}(u)\big| \geq \frac{\eta}{2}\right)\\
\leq & \frac{4}{\eta^2} \bE \left[\bar{M}_j^{(c)}(t)^2 \right] = \frac{4t}{\eta^2} \left[ \frac{1}{c}\bE (a_j(1)-\bar{a}_j)^2\right]\xrightarrow[c\rightarrow\infty]{} 0.
\end{align}

The process $\bar{N}^{(c)}_j(t)$ is a Martingale with summands bounded by some constant $K_2/c$, so by the Azuma-Hoeffding Inequality
\begin{equation}
\bP (\big| \bar{N}^{(c)}_j(t) \big| \geq \eta ) \leq 2 e^{- \frac{c\eta^2}{2 t K_2} }
\end{equation}
Again applying Doob's $L_2$ inequality
\begin{align}\label{tightN}
&\bP\Bigg( \sup_{\substack{u,v: u,v<t\\|v-u|<\delta}}  \Big| \bar{N}_j^{(c)}(v)-\bar{N}_j^{(c)}(u) \Big|\geq \eta\Bigg) \\
\leq &\bP \left( \sup_{u: u <t}\big|  \bar{N}_j^{(c)}(u) \big| \geq \frac{\eta}{2}\right)\\
\leq & \frac{4}{\eta^2} \bE \left[\bar{N}_j^{(c)}(t)^2 \right] \leq \left( \frac{4 K_2t}{\eta c}\right)^2\xrightarrow[c\rightarrow \infty]{}0.
\end{align}
Thus our sequence of processes $\{\bar{Q}^{(c)}\}_{c}$ are tight.

Note the inequalities above, \eqref{tightM} and \eqref{tightN}, also prove that the sequences $\bar{N}^{(c)}$ and $\bar{M}^{(c)}$ converge in distribution to zero.

It remains to show that any fluid limit process satisfies the fluid equations \eqref{FluidEqns}. By the Skorohod Representation Theorem \cite[Theorem C.8]{Ro10}] and since $\{\bar{Q}^{(c)}\}_{c}$ is tight: for any convergent subsequence of we can choose a subsequence where convergence occurs almost surely, along an appropriately chosen probability space. Thus, using  this and the definition of $\bar{S}_j$
\begin{align}
&\lim_{c\rightarrow\infty} \bar{Q}^{(c)}_j(t) \\
=&  \lim_{c\rightarrow\infty} \Big\{ \bar{Q}^{(c)}_j(0) + \bar{a}_j\frac{ \lfloor ct \rfloor}{c} - \bar{S}^{(c)}_j(t) + \bar{M}^{(c)}_j(t) + \bar{N}^{(c)}_j(t) \Big\}\\
=& q_j(0) + \bar{a}_jt - \lim_{c\rightarrow\infty} \bar{S}^{(c)}_j(t) \\
=& q_j(0) + \bar{a}_jt - \lim_{c\rightarrow\infty} \int_0^{{{\lfloor ct\rfloor} }/{c}}  \hat{\sigma}_j({Q}^{(c)}_j(cs))ds
=q_j(t).
\end{align}
As $\sigma^*$ is bounded it is clear that $q_j(t)$ is Lipschitz continuous. If $q_j(t)>0$ then, by continuity, $q_j(s)>0$ for an open region $s\in (t-\delta,t+\delta)$. 
We have for this choice of $t$ and $s$
\begin{align}
&q_j(t)-q_j(s) \\
=&\bar{a}_j (t -s) - \lim_{c\rightarrow\infty} \int_{{{\lfloor cs\rfloor} }/{c}}^{{{\lfloor ct\rfloor} }/{c}}  \hat{\sigma}_j({Q}^{(c)}_j(cu) du\\
  =& \bar{a}_j (t -s) -  \int_{{{\lfloor cs\rfloor} }/{c}}^{{{\lfloor ct\rfloor} }/{c}} \lim_{c\rightarrow\infty} \hat{\sigma}_j({Q}^{(c)}_j(cu) du\\
=& \bar{a}_j (t -s) -  \int_{{{\lfloor cs\rfloor} }/{c}}^{{{\lfloor ct\rfloor} }/{c}}  \sigma^*_j(q_j(u)) du
\end{align}
As Lemma \ref{QZ} states $\hat{\sigma}_j(q_j)$ is continuous for $q_j>0$, we apply the bounded convergence theorem in the second equality above. Now, in addition to being tight, we have proven the fluid equations, \eqref{FluidEqns}, hold.

\end{proof}

\section{Other Lemmas}\label{Lemmas}

The following lemma assists our fluid limit proof, Proposition \ref{FluidLimThrm}. Recall our definitions of $\hat{\sigma}(q)$, \eqref{MW-qq}, and $\sigma^*(q)$, \eqref{MW-sq}. We let $\{q^{(c)}\}_c$ be any unbounded sequence in $\bZ_+^\mJ$ with $||q^{(c)}||_1=c$. From this, we define $\bar{q}^{(c)}= q^{(c)}/c$.
\begin{lemma}\label{QZ}
If $\bar{q}^{(c)}\rightarrow q$  as $c\rightarrow \infty$ then, for each $j\in\mJ$ with $q_j> 0$,
\begin{equation*}
\hat{\sigma}_j({q}^{(c)}) \rightarrow \sigma^*_j(q)
\end{equation*}
 as $c\rightarrow \infty$. 
\end{lemma}
\begin{proof}
We can take an appropriate subsequence and without loss of generality, we may assume that the bounded sequence $\{ \hat{\sigma}(q^{(c)}) \}_{c}$ converges to some value $\hat{\sigma}$. In addition, we take $M$ such that $M\geq \sigma_j$ for all $\sigma\in\mS$ and $j\in\mJ$ and define $\kappa\in\bR^\mJ$ by
\begin{equation*}
\kappa_j=
\begin{cases}
M, & \text{if } q_j > 0,\\
0, & \text{if } q_j=0.
\end{cases}•
\end{equation*}•
Notice, since $\bar{q}^{(c)}\rightarrow q$, 
\begin{equation}\label{mon-q}
{q}^{(c)}\geq \kappa
\end{equation}
\text{ eventually as }$c\rightarrow\infty$.

By definition
\begin{equation}
\sum_{j\in\mJ} g_j(\hat{\sigma}_j) q_j^{\alpha} \leq  \sum_{j\in\mJ} g_j({\sigma}^*_j(q)) q_j^{\alpha}. 
\end{equation}

By strict concavity: $g_j(ps^1_j + (1-p)s^2_j)q_j >p g_j(s^1_j)q_j + (1-p)g_j(s^2_j)q_j$, the optimal solution $\sigma^*(q)$ is unique for each component $j$ with $q_j>0$. We define $G^*$ by
\begin{equation}
G^*=  \max_{s\in\coS}  \sum_{j\in\mJ} g_j(s_j) q_j^{\alpha}
\end{equation}•
By definition $\sigma^*(q)$ is optimal for the above optimization. Also, notice the above optimum remains equal to $G^*$ if we maximize over $<\mS\wedge \kappa>$ instead of $\coS$. Thus
\begin{align}
G^* &= \max_{\sigma\in <\mS\wedge \kappa>} \sum_{j\in\mJ} g_j({\sigma}) q_j^{\alpha}\\
& = \lim_{c\rightarrow\infty} \max_{\sigma\in <\mS\wedge \kappa>} \sum_{j\in\mJ} g_j({\sigma}_j) \big( \bar{q}_j^{(c)} \big)^{\alpha}\\
 & \leq  \lim_{c\rightarrow\infty} \max_{\sigma\in <\mS\wedge q^{(c)} >} \sum_{j\in\mJ} g_j({\sigma}_j) \big( \bar{q}_j^{(c)}  \big)^{\alpha}\\
& = \lim_{c\rightarrow\infty}  \sum_{j\in\mJ} g_j\big(\hat{\sigma}_j(q^{(c)} ) \big) \big( \bar{q}_j^{(c)} \big)^{\alpha}\\
& = \sum_{j\in\mJ} g_j\big(\hat{\sigma}_j\big) q_j^{\alpha}
\end{align}
The first equality, above, holds optimum remains equal to $G^*$ if we maximize over $<\mS\wedge \kappa>$ instead of $\coS$; the second equality holds by continuity of the objective function and compactness of $<\!\mS\wedge \kappa\!>$; the next inequality holds by monotonicity \eqref{mon-q}; and the remaining equalities hold by definition and continuity.

Thus $\hat{\sigma}$ is optimal and thus $\hat{\sigma}_j=\sigma^*_j(q)$ for all $j$ with $q_j>0$. So, as this limit $\hat{\sigma}_j$ is unique for all subsequences the lemma must hold.
\end{proof}

\begin{lemma}\label{UIlemma}
For every $t$, the random variables $\{ \bar{Q}^{(c)}(t)\}_{c\in\bN}$ are uniformly integrable.
\end{lemma}
\begin{proof}
It is sufficient to prove that these random variables are bounded above by a sequence of $L_2$--bounded random variables. The queue size is bounded above by the number of arrivals, so, for all $t\geq 0$ 
\begin{equation}\label{QUI}
\sum_{j\in\mJ} \bar{Q}^{(c)}_j(t) \leq  1 + \sum_{j\in\mJ} \sum_{s=1}^{\lceil ct \rceil} \frac{a_j(s)}{c}.
\end{equation}
By assumption, $a_j(s)$ has finite variance; thus, the right-hand side of inequality \eqref{QUI} is $L_2$-bounded; and so, $\sum_{j\in\mJ} \bar{Q}^{(c)}_j(t)$ is uniformly integrable. 
\end{proof}

\section*{Acknowledgment}
The author would like to thank Devavrat Shah and Yuan Zhong for their comments on this article.

\bibliographystyle{amsplain}
\bibliography{MW-AB}

\end{document}